\documentclass[pdflatex,sn-mathphys-num]{sn-jnl}

\usepackage{graphicx}%
\usepackage{multirow}%
\usepackage{amsmath,amssymb,amsfonts}%
\usepackage{amsthm}%
\usepackage{mathrsfs}%
\usepackage[title]{appendix}%
\usepackage{xcolor}%
\usepackage{textcomp}%
\usepackage{manyfoot}%
\usepackage{booktabs}%
\usepackage{algorithm}%
\usepackage{algorithmicx}%
\usepackage{algpseudocode}%
\usepackage{listings}%

\theoremstyle{thmstyleone}%
\newtheorem{theorem}{Theorem}
\newtheorem{lemma}{Lemma}
\newtheorem{corollary}{Corollary}

\theoremstyle{thmstyletwo}%

\theoremstyle{thmstylethree}%
\newtheorem{assumption}{Assumption}

\DeclareMathOperator*{\argmin}{arg\,min}

\begin{document}

\title[A federated Kaczmarz algorithm]{A federated Kaczmarz algorithm}


\author[1]{\fnm{Halyun} \sur{Jeong}}\email{hjeong2@albany.edu}

\author[2]{\fnm{Deanna} \sur{Needell}}\email{deanna@math.ucla.edu}

\author*[2]{\fnm{Chi-Hao} \sur{Wu}}\email{chwu93@math.ucla.edu}

\affil[1]{\orgdiv{Department of Mathematics \& Statistics}, \orgname{The State University of New York at Albany}, \orgaddress{\city{Albany}, \postcode{12222}, \state{NY}, \country{USA}}}

\affil*[2]{\orgdiv{Department of Mathematics}, \orgname{University of California Los Angeles}, \orgaddress{\city{Los Angeles}, \postcode{90095}, \state{CA}, \country{USA}}}



\abstract{In this paper, we propose a federated algorithm for solving large linear systems that is inspired by the classic randomized Kaczmarz algorithm. We provide convergence guarantees of the proposed method, and as a corollary of our analysis, we provide a new proof for the convergence of the classic randomized Kaczmarz method. 
We demonstrate experimentally the behavior of our method when applied to related problems. For underdetermined systems, we demonstrate that our algorithm can be used for sparse approximation. For inconsistent systems, we demonstrate that our algorithm converges to a horizon of the least squares solution. Finally, we apply our algorithm to real data and show that it is consistent with the selection of Lasso, while still offering the computational advantages of the Kaczmarz framework and thresholding-based algorithms in the federated setting.}




\maketitle

\section{Introduction}\label{introduction}
In this paper, we propose a federated algorithm for solving large linear systems. Federated learning was originally proposed by \cite{mcmahan2017communication} to train neural networks in a decentralized setting. The global model is trained across multiple clients (e.g., mobile devices, sensors, or edge nodes) without transferring local data to a central server. This method improves privacy, reduces communication overhead, and enables learning from heterogeneous, distributed datasets; for instance, see \cite{WETAL21} for more details. On the other hand, the Kaczmarz algorithm \cite{karczmarz1937angenaherte} is an iterative method for solving overdetermined linear systems, and the randomized Kaczmarz algorithm (RK) \cite{strohmer2009randomized} is a version with a particular sampling scheme. Given a matrix $A\in\mathbb{R}^{m\times n}$, we denote by $a_j$ the $j$-th row of $A$, and we denote the Frobenius norm,
\[
\Vert A\Vert_F^2 = \sum_{j=1}^m \Vert a_j\Vert_2^2.
\]
To solve a linear system $Ax=b$ for $A\in\mathbb{R}^{m\times n}$, $b\in\mathbb{R}^m$, RK considers the iteration,
\[
x_{k+1} = x_k + \frac{b_{j_k} - \langle a_{j_k}, x_k\rangle}{\Vert a_{j_k}\Vert^2}\,a_{j_k}, \quad k = 0, 1, \dotsc,
\]
where $j_k$'s are independent identically distributed (i.i.d.) random variables with distribution
$\mathbb{P}(j_k = j) = \Vert a_j\Vert_2^2/\Vert A\Vert_F^2$ for $k=0,1,\dotsc$. For consistent overdetermined linear systems, it was shown in \cite{strohmer2009randomized} that RK converges linearly to the solution in expectation. In a broader framework, it is also known that RK can be viewed as a stochastic gradient descent method with carefully chosen step sizes \cite{needell2014stochastic}. 

We consider the federated setting where there are $M$ local clients, and each client owns a subset of equations. In particular, let $Ax=b$ be the overdetermined, consistent linear system we aim to solve, where $A\in \mathbb{R}^{m\times n}$ and $b\in\mathbb{R}^m$. The system is partitioned into $M$ parts,
\[
A = \begin{bmatrix}
A_1 \\ \vdots \\ A_M
\end{bmatrix} \quad\mbox{and}\quad
b = \begin{bmatrix}
b_1 \\ \vdots \\ b_M
\end{bmatrix},
\]
and the $i$-th client sees the system $A_i x = b_i$ for $i\in [M]$; note that locally the system $A_i x = b_i$ may be underdetermined for $i\in [M]$. We briefly describe the framework of federated algorithms using federated averaging (FedAvg) \cite{mcmahan2017communication} as an example. At each federated round, a subset of local clients is selected, and the global server broadcasts the current global parameters (position) to the selected local clients. Each of the selected clients performs local updates using the global position as the initial and returns the updated local position after some iterations. The global server averages the local updates returned by the local clients and uses the averaged position as the updated global position. The process is iterated until it converges. Following the notation used in the federated optimization community \cite{WETAL21}, we denote by $x^{(t)}\in\mathbb{R}^n$ the solution at the global server after $t$ federated rounds, and $x_i^{(t,k)}$ the solution at the $i$-th client after $k$ local iterations during the $t+1$-th federated round. We propose Algorithm~\ref{algorithm1} for solving large linear systems in a federated setting.

We briefly comment on the ideas behind Algorithm~\ref{algorithm1} and the difficulties for the convergence analysis. In our setting, we allow the linear systems at the local clients $A_ix = b_i$, to be underdetermined for $i\in \{1,\dotsc, M\}$, and it is known that RK converges to the orthogonal projection of the initial position onto the affine subsets
\begin{equation}\label{introeq2}
C_i = \left\{x\in\mathbb{R}^n: A_ix = b_i \right\}, \quad i = 1,\dotsc, M
\end{equation}
for underdetermined linear systems; in particular, if the initial position is orthogonal to the null space of $A_i$, RK converges to the least norm solution of $A_ix=b_i$; see \cite{ma2015convergence}. Based on this observation, we see that $\Delta_i^{(t)}$ defined in Algorithm~\ref{algorithm1} is approximating a normal vector to a hyperplane containing $C_i$ (up to a normalizing constant), and projecting $x^{(t)}$ onto $C_i$'s is approximately the same as projecting $x^{(t)}$ onto
\[
\tilde{C}_i = \left\{x\in\mathbb{R}^n: \Delta_i^{(t)} x = d_i \right\}, \quad i=1,\dotsc, M.
\]
Therefore, instead of treating the local model changes $\Delta_i^{(t)}$'s as simply displacements, we transform them into approximate hyperplanes $\tilde{C}_i$'s, which we then use at the server to run RK. In some sense, Algorithm~\ref{algorithm1} is a stochastic process where at each time $t\in \{0,1,\dotsc, T-1\}$, we choose a collection of affine subsets, dependent on $x^{(t)}$, onto which to project. Given that it is difficult to study the convergence with a purely algebraic approach, we take a more geometric approach; see Section~\ref{mainresults} for details.  

\begin{algorithm}\label{algorithm1}
    \caption{FedRK}\label{your_label}
    \begin{algorithmic}
        \State  Initial model $x^{(0)}$
        \For{t = $0$ to $T-1$} 
            \State Sample a subset $S^{(t)}$ of clients
            \For{client $i\in S^{(t)}$}
                \State Initialize local model $x_i^{(t,0)} = x^{(t)}$
                \For{$k = 0,\dotsc, \tau-1$ }
                    \State Perform RK initialized with $x^{(t, k)}$
                \EndFor
                \State Compute the local model changes $\Delta_i^{(t)} = x_i^{(t,\tau)} - x_i^{(t,0)}$
            \EndFor
            \For{$i\in S^{(t)}$}
                \State Define $d_i = \langle\Delta_i^{(t)}, \Delta_i^{(t)} + x^{(t)} \rangle$
            \EndFor
            \State Apply RK to solve $\Delta^{(t)} x = d$ (ignoring the rows where $\Delta_i^{(t)} = 0$ with uniform sampling scheme), and let $x^{(t+1)}$ be the solution after $\tau_g$ iterates 
            
        \EndFor
    \end{algorithmic}
\end{algorithm}

For the convergence analysis, we consider two scenarios. First, we study the scenario where the local clients run finitely many iterations and the server runs one iteration ($\tau_g = 1$ in Algorithm~\ref{algorithm1}). Second, we study the scenario where the local clients run infinitely many iterations ($\tau = \infty$ in Algorithm~\ref{algorithm1}), and the server runs some finite number of iterations. We state the assumptions that we make in our main theorems, except for Corollary~\ref{introcor1}, where the system is allowed to be underdetermined.
\begin{assumption}\label{introassumption1}
The linear system $Ax = b$ is overdetermined.
\end{assumption}
\noindent By an appropriate translation, one can assume that the true solution $x^*=0$, and thus it is natural to also include the second assumption:
\begin{assumption}\label{introassumption2}
The solution to the linear system is $x^*=0$.
\end{assumption}

\noindent The proof for the first scenario ($\tau_g=1$) utilizes that of RK applied to a suitable linear system. Specifically, we show the following: 

\begin{theorem}\label{introthm2}
Following the notation in Algorithm~\ref{algorithm1}, assume $\tau = T$, $\tau_g = 1$ and $\vert S^{(t)}\vert \equiv N\in \mathbb{N}$. Let $X^{(t+1)}$ be the global update after $t+1$ iterations. Then there exists $0< \beta < 1$ such that
\[
\mathbb{E}\Vert X^{(t+1)}\Vert^2 \leq \beta^{t+1} \Vert X^{(0)}\Vert^2, \quad t= 0,1,\dotsc.
\]
\end{theorem}

\noindent For the second scenario, let us first assume the server performs one RK iteration ($\tau_g=1$ in Algorithm~\ref{algorithm1}). Since we assume that the solution $x^*=0$, the $C_i$'s defined in (\ref{introeq2}) are linear subspaces for $i\in [M]$. Denote by $P_{C_i}$ the orthogonal projection operator onto $C_i$ for $i\in [M]$. We see that (following the notation in Algorithm~\ref{algorithm1}),
\begin{equation}\label{introeq1}
\mathbb{E}X^{(t+1)} = \sum_{s\in S^{(t)}} \frac{1}{\vert S^{(t)}\vert} \left(I - \frac{(X^{(t)}-P_{C_s}X^{(t)})(X^{(t)}-P_{C_s}X^{(t)})^T}{\Vert X^{(t)} - P_{C_s}X^{(t)}\Vert^2} \right) X^{(t)},  
\end{equation}
defines the sequence produced by our federated Kaczmarz algorithm. We first prove a technical theorem that characterizes the decay of Algorithm~\ref{algorithm1} when the current position is one unit distance away from the true solution, which involves studying a function related to (\ref{introeq1}); see Theorem~\ref{maintheorem1}. Then we use the technical theorem to prove the following:

\begin{theorem}\label{introthm1}
Following the notation in Algorithm~\ref{algorithm1}, assume $\tau = \infty$, $\tau_g = T$ and $\vert S^{(t)}\vert \equiv N\in \mathbb{N}$. Let $X^{(t+1)}$ be the global update after $t+1$ iterations. Then there exists $0< \beta < 1$ such that
\[
\mathbb{E}\Vert X^{(t+1)}\Vert^2 \leq \beta^{t+1} \Vert X^{(0)}\Vert^2, \quad t= 0,1,\dotsc.
\]
\end{theorem}

\noindent Perhaps interestingly, Theorem~\ref{maintheorem1} also gives an alternative proof for the convergence of classical RK; see Corollary~\ref{maincor1}. 
Finally, in the second scenario, we also deduce that Algorithm~\ref{algorithm1} converges linearly to the orthogonal projection from the initial global position $x^{(0)}$ onto $\cap_{i=1}^M C_i$ when the whole system is underdetermined. Specifically, if $\cap_{i=1}^M C_i$ is a single point, Algorithm~\ref{algorithm1} converges to the true solution as before.

\begin{corollary}\label{introcor1}
Following the notation in Algorithm~\ref{algorithm1}, assume $\tau = \infty$, $\tau_g = T$ and $\vert S^{(t)}\vert \equiv N\in \mathbb{N}$. Let $X^{(t+1)}$ be the global update after $t+1$ iterations. Denote $C=\cap_{i=1}^M C_i$, where $C_i$'s are defined in (\ref{introeq2}), and $P_C$ the orthogonal projection onto $C$. Then there exists $0< \beta < 1$ such that
\[
\mathbb{E}\Vert X^{(t+1)} - P_CX^{(0)}\Vert^2 \leq \beta^{t+1} \Vert X^{(0)}\Vert^2, \quad t= 0,1,\dotsc.
\]
\end{corollary}

We demonstrate experimentally the behavior of FedRK when applied to related problems. For sparse approximation problems, the system is modeled as $b = Ax^* + e$ where $A\in\mathbb{R}^{m\times n}$ is a wide matrix ($m < n$), $x^* \in\mathbb{R}^n$ is a true sparse signal, and $e\in\mathbb{R}^m$ is noise. Given a sparsity level $s\in\mathbb{N}$, the hard thresholding operator $T_s:\mathbb{R}^n \rightarrow \mathbb{R}^n$ is defined as the orthogonal projection onto the entries with the $s$ largest magnitudes. By combining with a hard thresholding operator, the Kaczmarz algorithm has been proposed as a method to solve such problems; see for instance \cite{jeong2023linear, zhang2015iterative}. In the federated setting, we show that our algorithm can be combined with the hard thresholding operator $T_s$ to solve the sparse approximation problem; see Algorithm~\ref{algorithm3}. For the least squares problem, where $A\in\mathbb{R}^{m\times n}$ is a tall matrix ($m > n$) and the goal is to minimize $\Vert Ax-b\Vert^2$, it is known that the randomized Kaczmarz algorithm does not converge to the least squares solution in general, but its iterates reach within a horizon of the solution; see \cite{needell2010randomized}. In the federated setting, we show that our algorithm has the same behavior. Moreover, we show that by adding a suitable amount of noisy columns to $A$, one can shrink the horizon. Finally, we apply Algorithm~\ref{algorithm3} to the prostate cancer data considered in \cite{tibshirani1996regression}, and we demonstrate that the selection of Algorithm~\ref{algorithm3} is consistent with Lasso. 

\begin{algorithm}[ht]
    \caption{FedRK with thresholding}\label{algorithm3}
    \begin{algorithmic}
        \State  Initial model $x^{(0)}$
        \For{t = $0$ to $T-1$} 
            \State Sample a subset $S^{(t)}$ of clients
            \For{client $i\in S^{(t)}$}
                \State Initialize local model $x_i^{(t,0)} = x^{(t)}$
                \For{$k = 0,\dotsc, \tau-1$ }
                    \State Perform RK initialized with $x^{(t, k)}$
                \EndFor
                \State Compute the local model changes $\Delta_i^{(t)} = x_i^{(t,\tau)} - x_i^{(t,0)}$
            \EndFor
            \For{$i\in S^{(t)}$}
                \State Define $d_i = \langle\Delta_i^{(t)}, \Delta_i^{(t)} + x^{(t)} \rangle$
            \EndFor
            \State Apply RK to solve $\Delta^{(t)} x = d$ (ignoring the rows where $\Delta_i^{(t)} = 0$ with uniform sampling scheme), and let $x^{(t+1)}$ be the solution after $\tau_g$ iterates \\

            \State Apply hard thresholding operator $x^{(t+1)} = T_s x^{(t+1)}$
        \EndFor
    \end{algorithmic}
\end{algorithm}

\subsection{Contribution}

We summarize our contributions. First, we propose the federated Kaczmarz algorithm (FedRK) for solving large linear systems in the federated setting, and we prove the linear convergence of our algorithm. As a corollary of our analysis, we give an alternative proof for the linear convergence of the classic randomized Kaczmarz algorithm (RK). Second, we demonstrate experimentally that our algorithm can be combined with hard thresholding to solve sparse approximation problems. We also demonstrate experimentally that it converges to a horizon of the least squares solution as the RK when applied to inconsistent systems. Finally, we apply our algorithm to real data and show the possible use for feature selection in the federated setting.  

\subsection{{Organization}}

The rest of the paper is organized as follows. In Section~\ref{mainresults}, we present our main results. We first present Theorem~\ref{maintheorem1}, which characterizes the convergence behavior when a collection of local clients is selected and the current position is one unit distance away from the true solution; the proof of Theorem~\ref{maintheorem1} can be interesting in its own right. In particular, a corollary of Theorem~\ref{maintheorem1} is an alternative proof for the linear convergence of the classic RK; see Corollary~\ref{maincor1}. We then prove the convergence theorems of our algorithm. The proofs of Theorem~\ref{introthm2} and Theorem~\ref{introthm1} are based on slightly different strategies, and we provide roadmaps of the proofs as guides. In Section~\ref{experiments}, we present the experiments. We demonstrate experimentally the linear convergence of Algorithm~\ref{algorithm1}. Perhaps interestingly, our experiment shows that running more local iterations helps the algorithm converge faster, which is usually not seen in other federated algorithms. We then apply Algorithm~\ref{algorithm3} to sparse approximation problems and Algorithm~\ref{algorithm1} to least squares problems for inconsistent systems; our results show that our algorithm behaves similarly to the classical RK, which hints that our algorithm can be efficiently combined with other Kaczmarz variants to extend other variants to the federated setting.  Finally, we apply Algorithm~\ref{algorithm3} to real data and show that our algorithm can potentially be used for feature selection in the federated setting. In Section~\ref{discussion}, we present some discussion and future directions.  

\section{Main Results}\label{mainresults}
\subsection{Notation}\label{mainsection1}
In this section, we define the notation used throughout. Denote by $[M] = \{1,\dotsc, M\}$. We partition the system into $M$ parts,
\[
A = \begin{bmatrix}
A_1 \\ \vdots \\ A_M
\end{bmatrix} \quad\mbox{and}\quad
b = \begin{bmatrix}
b_1 \\ \vdots \\ b_M
\end{bmatrix},
\]
and the $i$-th client sees the system $A_i x = b_i$ for $i\in [M]$. Let $C_i$ be defined as in (\ref{introeq2}) for $i\in [M]$. Under Assumption~\ref{introassumption2}, we have $C_i\subseteq \mathbb{R}^n$ are linear subspaces for $i\in [M]$. We denote by $P_{C}: \mathbb{R}^n \rightarrow C$ the orthogonal projection onto a linear subspace $C\subseteq \mathbb{R}^n$. 

Let $\mathcal{P}$ be the set of probability measures on $[M]$. Given $x\in\mathbb{R}^n$ and $S\subseteq [M]$, we denote by $\mathcal{P}_S\subseteq \mathcal{P}$ the set of probability measures such that
\begin{equation}\label{maineq1}
\mathcal{P}_S = \left\{ p\in\mathcal{P}: \operatorname{supp}(p) = S\right\};
\end{equation}
for instance, the uniform distribution over $S\subseteq [M]$ is in $\mathcal{P}_S$. An object that forms the core of our analysis is the following function,
\begin{equation}\label{maineq3}
f(S,\,x,\,p, \,y) = \sum_{s\in S} p(s) \left\Vert \left(I - \frac{(x-P_{C_s}x)(x-P_{C_s}x)^T}{\Vert x - P_{C_s}x\Vert^2} \right) y\right\Vert^2
\end{equation}
for $S\subseteq [M]$, $x\in \cap_{s\in S} C_s^c$, $p\in\mathcal{P}$ and  $y\in \mathbb{R}^n$. For $p\in\mathcal{P}_S$, this function measures the average decrease of the norm of $y$, when randomly projecting onto the hyperplanes formed by the normal vectors $\{x-P_{C_s}x\}_{s\in S}$. In fact, we will study the function on a refined domain; see Section~\ref{mainsection2} for more details.

For the readers' convenience, we review some notions in geometry and topology; one can find detailed descriptions in \cite{munkrestopology} and \cite{spivakgeometry}, for instance. A topological space is a pairing $(X, \mathcal{T})$, where $X$ is the whole space and $\mathcal{T}$ is the topology, i.e. a collection of open subsets satisfying
\begin{itemize}
\item $\emptyset\in\mathcal{T}$
\item $U\in \mathcal{T}$ implies $U^c \in \mathcal{T}$
\item $\{U_{\alpha}\}_{\alpha\in A}\subseteq \mathcal{T}$ implies $\cup_{\alpha \in A} U_{\alpha} \in \mathcal{T}$, where $A$ is an arbitrary index set (possibly uncountable)
\end{itemize}
A set $K\subseteq X$ is compact if any open covering of $K$ admits a finite sub-covering. A map $f: X\rightarrow Y$ between two topological spaces is continuous if $f^{-1}(U)\subseteq X$ is open for all open sets $U\subseteq Y$. If $f: X\rightarrow Y$ between two topological spaces is continuous, then $K$ is compact implies $f(K)$ is compact. Another ingredient in our proof is the implicit function theorem, which we recall the statement:

\begin{theorem}[Implicit function theorem]
Let $U\subseteq \mathbb{R}^{n+m}$ be an open subset, and $f: U\rightarrow \mathbb{R}^m$ be smooth. Given $(x_0, y_0)\in\mathbb{R}^{n+m}$ such that $f(x_0, y_0) = 0$ and $D_yf(x_0, y_0): \mathbb{R}^m \rightarrow \mathbb{R}^m$ is invertible, then, in a open neighborhood of $(x_0, y_0)$, the level set 
\[
\left\{(x, y)\in\mathbb{R}^{n+m}: f(x, y) = 0\right\}
\]
is smoothly parameterized by $x$; i.e. in a neighborhood $V\subseteq\mathbb{R}^n$ of $x_0$ there exists a smooth function $g: V\rightarrow \mathbb{R}^m$ such that
\[
\left\{(x, y)\in\mathbb{R}^{n+m}: f(x, y) = 0, \,x\in V\right\} = \left\{(x, g(x))\in\mathbb{R}^{n+m}: x\in V\right\}.
\]
\end{theorem}

\subsection{A technical theorem}\label{mainsection2}
In this section, we prove our main technical theorem, and deduce the convergence of the classical randomized Kaczmarz as a corollary. To motivate the setting, we start with a discussion of how one can reduce the analysis to a function on the product of two spheres.

Given $x\in\mathbb{R}^n\setminus \{0\}$, one has by linearity
\[
P_{C_i}x = \Vert x\Vert P_{C_i}\frac{x}{\Vert x\Vert},
\]
and
\[
x - P_{C_i}x = \Vert x\Vert \left( \frac{x}{\Vert x\Vert} - P_{C_i}\frac{x}{\Vert x\Vert}\right).
\]
This shows that the normal vector obtained after normalizing $x-P_{C_i}x$ is independent of the length of $x\in\mathbb{R}^n$. Similarly, for $y\in\mathbb{R}^n\setminus \{0\}$, one has
\[
\left(I - \frac{(x-P_{C_s}x)(x-P_{C_s}x)^T}{\Vert x - P_{C_s}x\Vert^2} \right) y =  \Vert y\Vert \left(I - \frac{(x-P_{C_s}x)(x-P_{C_s}x)^T}{\Vert x - P_{C_s}x\Vert^2} \right) \frac{y}{\Vert y\Vert}.
\]
This suggests that it is natural to study our problem on the unit sphere $S^{n-1}\subseteq \mathbb{R}^n$. Given $\epsilon>0$, denote 
\begin{equation}\label{maineq5}
B_S = \left\{ x\in S^{n-1}: \operatorname{dim}\left( \operatorname{span}\{x-P_{C_s}x\}_{s\in S} \right) < \vert S\vert \right\}, \quad B_{S,\,\epsilon} = \{x\in S^{n-1}: d(x, B_S) < \epsilon\}
\end{equation}
and $\overset{*}{\mathcal{S}}_{S, \,\epsilon} = S^{n-1}\setminus B_{S,\,\epsilon}$.
We will consider $f$ as a function on $2^M\times \overset{*}{\mathcal{S}}_{S,\,\epsilon}\times \mathcal{P}\times S^{n-1}$ for technical reasons that will become clear later; see Figure~\ref{mainfig1} for an illustration. It is clear that $f(S,\, \cdot,\, p,\, \cdot): \overset{*}{\mathcal{S}}_{S,\,\epsilon}\times S^{n-1}\rightarrow \mathbb{R}$ is smooth for each $(S,\,p)\in 2^M\times \mathcal{P}$. We have the following:

\begin{theorem}\label{maintheorem1}
Given $S\subseteq [M]$ and $\epsilon>0$, let
\[
h_{S,\,x} = \bigcap_{s\in S} \left\{ y\in S^{n-1}: (x-P_{C_s}x)^T\, y = 0 \right\}, \quad x\in S^{n-1}.
\]
Denote
\[
C_{S,\,\epsilon} = \{(x, y)\in \overset{*}{\mathcal{S}}_{S,\,\epsilon}\times S^{n-1}: d(y,\,h_{S,\,x}) < \epsilon \},\quad D_{S,\,\epsilon} = (\overset{*}{\mathcal{S}}_{S,\,\epsilon}\times S^{n-1}) \setminus C_{S,\,\epsilon}.
\]
Given $p\in \mathcal{P}_S$ defined in (\ref{maineq1}),  consider $f(S,\,\cdot,\,p,\,\cdot): D_{S,\,\epsilon} \rightarrow \mathbb{R}$ defined as
\[
f(S,\,x,\,p,\,y) = \sum_{s\in S} p(s) \left\Vert \left(I - \frac{(x-P_{C_s}x)(x-P_{C_s}x)^T}{\Vert x - P_{C_s}x\Vert^2} \right) y\right\Vert^2.
\]
Then, $f(S,\,x,\,p,\,y) \leq \alpha(S,\,p,\,\epsilon) < 1$ for $(x,y)\in D_{S,\,\epsilon}$.
Moreover, for all $(x,y)\in D_{S,\,\epsilon}$, there exists $s(y)\in S$ such that
\[
\left\Vert \left(I - \frac{(x-P_{C_{s(y)}}x)(x-P_{C_{s(y)}}x)^T}{\Vert x - P_{C_{s(y)}}x\Vert^2} \right) y\right\Vert^2 \leq \frac{\alpha(S,\,p,\,\epsilon)}{\vert S\vert\, p(s(y))}.
\]
\end{theorem}

\begin{proof}
To make our presentation clear, we can assume that the index set, $S=\{ 1,\dotsc, \vert S\vert\}$. Our first objective is to show that $C_{S,\,\epsilon}\subseteq \overset{*}{\mathcal{S}}_{S,\,\epsilon}\times S^{n-1}$ is open. 

Given $(x_0, y_0)\in \overset{*}{\mathcal{S}}_{S,\,\epsilon}\times S^{n-1}$ such that $x_0\in B_{S,\,\epsilon}$ and $y_0\in h_{S,\,x_0}$, one can find $a_1, \dotsc, a_{n-\vert S\vert-1}$ such that
\[
\left\{x_0-P_{C_1}x_0,\dotsc, x_0-P_{C_{\vert S\vert}}x_0, a_1,\dotsc, a_{n-\vert S\vert-1}, y \right\} \subseteq \mathbb{R}^n
\]
is a basis. One can consider $F: \mathbb{R}^{2n}\rightarrow \mathbb{R}^n$ defined as
\[
F(x, y) = \begin{bmatrix}
(x- P_{C_1}x)^Ty \\
\vdots \\
(x- P_{C_\vert S\vert}x)^Ty \\[3pt]
a_1^Ty \\
\vdots \\
a_{n-\vert S\vert -1}^Ty \\[3pt]
\sum_{i=1}^n y_i^2 -1 \\
\end{bmatrix}.
\]
Denote by $DF: \mathbb{R}^{2n}\rightarrow \mathbb{R}^{n\times 2n}$ the Jacobian matrix of $F$, and denote by $D_xF$, $D_yF$ the submatrices when the partial derivatives are only taken with respect to $x$, $y$, respectively. A direct calculation shows
\[
DF(x_0, y_0) = \begin{bmatrix} D_xF(x_0, y_0) & D_yF(x_0, y_0)
\end{bmatrix} = \begin{bmatrix}
y^T(I - P_{C_1}) & x^T(I - P_{C_1})^T \\
\vdots & \vdots \\
y^T(I - P_{C_{\vert S\vert}}) & x^T(I - P_{C_{\vert S\vert}})^T \\
0 & a_1^T \\
\vdots  & \vdots\\
0 & a_{n-\vert S\vert -1}^T \\[3pt]
0 & 2y^T
\end{bmatrix}\in \mathbb{R}^{n\times 2n},
\]
and it is clear from our choice that $D_yF(x_0, y_0)$ is of full rank. If one considers the level set 
\[
\left\{ (x, y)\in\mathbb{R}^{2n}: F(x,y) = 0 \right\},
\]
then there exists an open ball $B(x_0)\subseteq \mathbb{R}^n$ such that $y$ is a function of $x$ and $y(x_0) = y_0$ by the implicit function theorem. 

Given $(x_0,y_0)\in C_{S,\,\epsilon}$, there exists $y_*\in h_{S,\,x_0}$ such that
\[
d(y_0, y_*) = d(y_0, h_S(x_0)) = \epsilon' < \epsilon,
\]
since $\{y_0\}\subseteq S^{n-1}$ is compact and $h_S(x_0)\subseteq S^{n-1}$ is closed. By the argument above, there exists $\delta_1 > 0$, and a diffeomorphism $g:B_{\delta_1}(x_0)\rightarrow g(B_{\delta_1}(x_0))$ such that $g(x_0) = y_*$ and
\[
F(x, g(x)) = 0, \quad x\in B_{\delta_1}(x_0);
\]
i.e., $y$ is a function of $x$ in a neighborhood of $x_0$. By continuity of $g:B_{\delta_1}(x_0)\rightarrow g(B_{\delta_1}(x_0))$, there exists $\delta_2>0$ such that 
\[
d(x, x_0) < \delta_2 \quad\Rightarrow\quad d(g(x), g(x_0)) < \frac{\epsilon - d(y_0, y_*)}{4}.
\]
Define $\delta_3 = (\epsilon - d(y_0, y_*))/4$, and let $\delta = \min\{\delta_1, \delta_2, \delta_3\}$. We claim that $B_{\delta}(x_0)\times B_{\delta}(y_0) \subseteq C_{S,\,\epsilon}$, and so $(x_0, y_0)\in \overset{*}{\mathcal{S}}_{S,\,\epsilon}\times S^{n-1}$ is an interior point. For $(x,y)\in B_{\delta}(x_0)\times B_{\delta}(y_0)$,
\[
d(y, g(x)) \leq d(y, y_0) + d(y_0, y_*) + d(y_*, g(x)) < \epsilon.
\]
Finally, one has
\[
d(y, h_S(x)) \leq d(y, g(x)) < \epsilon,
\]
since $g(x)\in h_S(x)$, and this shows that $(x_0, y_0)\in \overset{*}{\mathcal{S}}_{S,\,\epsilon}\times S^{n-1}$ is an interior point. 

One has $D_{S,\,\epsilon}$ is compact, since $C_{S,\,\epsilon}\subseteq \overset{*}{S}_{\epsilon}\times S^{n-1}$ is open. One also has that $f(S,\,x,\,p,\, y) < 1$ for all $(x, y)\in D_{S,\,\epsilon}$, and therefore there exists $\alpha(S,\,p,\,\epsilon) < 1$ such that 
\[
f(S,\,x,\,p,\,y) \leq \alpha(S,\,p,\,\epsilon) < 1, \quad\forall (x,y)\in D_{S,\,\epsilon}.
\]
Finally, by pigeonhole principle, there exists $s(y)\in S$ such that 
\[
\vert S\vert\, p(s(y))\, \left\Vert \left(I - \frac{(x-P_{C_{s(y)}}x)(x-P_{C_{s(y)}}x)^T}{\Vert x - P_{C_{s(y)}}x\Vert^2} \right) y\right\Vert^2 \leq \alpha(S,\,p,\,\epsilon)
\]
for fixed $(x,y)\in D_{S,\,\epsilon}$. This concludes our proof.
\end{proof}

\begin{figure}[ht]
    \centering
    \includegraphics[width=119mm]{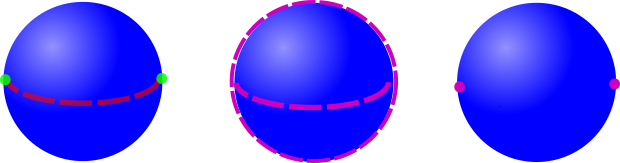}
    \caption{Suppose that  $C_1$ is a line intersecting the sphere at the green dots (left), and $C_2$ is a plane intersecting the sphere at the red (dashed) diametric circle. Then, the pink diametric (dashed) circles in the second sphere are the points we remove in the $x$-sphere, and the pink dots in the third sphere are the points we remove in the $y$-sphere mentioned in Theorem~\ref{maintheorem1}. More precisely, a neighborhood of those points is removed.} 
    \label{mainfig1}
\end{figure}

As a corollary, we demonstrate how Theorem~\ref{maintheorem1} implies the convergence of the classical randomized Kaczmarz algorithm. Indeed, the classical setting can be viewed as the scenario where one has $M=m$ local clients, and each local client has one equation of the linear system $Ax=b$; see \cite{huang2024randomized}. This corresponds to $S = [M]$ and $\{C_s\}_{s\in [M]}$ are $n-1$-dimensional linear subspaces (hyperplanes) in Theorem~\ref{maintheorem1}. 

Recall that we defined $\mathcal{P}_S$ in (\ref{maineq1}). Given a sampling scheme $p\in \mathcal{P}_{[M]}$, if we denote by $a_s$ the normal vector (unique up to sign) of $C_s$ for $s\in [M]$, the randomized Kaczmarz algorithm is defined by 
\begin{equation}\label{maineq4}
\mathbb{P} \left[ \left. Y_{k+1} = \left(I - a_s\, a_s^T \right) Y_k \right\vert Y_k \right] = p(s), \quad s\in [M].
\end{equation}
We have the following:

\begin{corollary}\label{maincor1}
Given an initial $y_0\in\mathbb{R}^n$ and a sampling scheme $p\in\mathcal{P}_{[M]}$ defined in (\ref{maineq1}), define $Y_1,\dotsc, Y_k,\dotsc$ by (\ref{maineq4}). Let $\alpha = \max_{y\in S^{n-1}} y^T \sum_{s\in [M]} p(s) \left(I - a_s\, a_s^T \right) y$.
Then  
\[
\mathbb{E}\left( \left. \Vert Y_{k+1}\Vert^2 \, \right\vert y_0\right) \leq \alpha^{k+1} \Vert y_0\Vert^2.
\]
\end{corollary}

\begin{proof}
First, we describe what $D_{S,\,\epsilon}$ in Theorem~\ref{maintheorem1} is in this scenario. By picking $\epsilon>0$ small enough, one has that $\overset{*}{\mathcal{S}}_{S,\,\epsilon}\neq \emptyset$. Let $\{a_s\}_{s\in [M]}$ be the normal vectors of the hyperplanes $\{C_s\}_{s\in [M]}$. One has 
\[
x - P_{C_s}x \in \operatorname{span}\{a_s\}, \quad\forall x\in \overset{*}{\mathcal{S}}_{S,\,\epsilon},\, s\in [M],
\]
and
\[
\bigcap_{s\in [M]} \left\{y\in \mathbb{R}^n: (x-P_{C_s}x)^T\, y = 0 \right\} = \bigcap_{s\in [M]} \left\{y\in \mathbb{R}^n: a_s^T\, y = 0 \right\} = \{0\}; 
\]
this implies that $h_{S,\,x} = S^{n-1}\cap \{0\} = \emptyset$, $C_{S,\,\epsilon} = \emptyset$ and therefore $D_{S,\,\epsilon} = \overset{*}{\mathcal{S}}_{S,\,\epsilon}\times S^{n-1}$.

Now, we study the convergence of the randomized Kaczmarz algorithm. By (\ref{maineq4}), 
\begin{align*}
\mathbb{E}\left( Y_{k+1} \vert Y_k\right) &= \sum_{s\in [M]} p(s) \left\Vert \left(I - a_s\, a_s^T \right) Y_k \right\Vert^2 \\[3pt]
&= \Vert Y_k\Vert^2\sum_{s\in [M]} p(s) \left\Vert \left(I - a_s\, a_s^T \right) \frac{Y_k}{\Vert Y_k\Vert} \right\Vert^2.
\end{align*}
Pick some $x \in \overset{*}{\mathcal{S}}_{S,\,\epsilon}$. We recognize the summation above is equal to $f([M],\,x,\,p,\, Y_k/\Vert Y_k\Vert)$. By Theorem~\ref{maintheorem1}, the function, $f([M],\,\cdot,\, p,\,\cdot): D_{S,\,\epsilon}\rightarrow\mathbb{R}$, satisfies
\[
f([M],\,x,\,p,\,y) = \sum_{s\in [M]} p(s) \left\Vert \left(I - a_s\,a_s^T \right) y\right\Vert^2 \leq \alpha([M],\,p,\,\epsilon) < 1, \quad\forall y\in S^{n-1}.
\]
(Note: the function is independent of $x\in \overset{*}{\mathcal{S}}_{S,\,\epsilon}$ here). Therefore,
\[
\mathbb{E}\left( Y_{k+1} \vert Y_k\right) = \Vert Y_k\Vert^2\,f([M],\,x,\,p,\, Y_k/\Vert Y_k\Vert) \leq \Vert Y_k\Vert^2\,\alpha([M],\,p,\,\epsilon)
\]
by which one can iterate to conclude $\mathbb{E}\left( \left. \Vert Y_{k+1}\Vert^2 \, \right\vert y_0\right) \leq \alpha^{k+1} \Vert y_0\Vert^2$. 

Finally, a closer look shows that
\[
\alpha([M],\,p,\,\epsilon) = \max_{y\in S^{n-1}}\sum_{s\in [M]} p(s) \left\Vert \left(I - a_s\,a_s^T \right) y\right\Vert^2 = \max_{y\in S^{n-1}} y^T \sum_{s\in [M]} p(s) \left(I - a_s\, a_s^T \right) y,
\]
which gives the variational formula for the constant $\alpha$. This concludes our proof.
\end{proof}

In Corollary~\ref{maincor1}, we showed the convergence of the randomized Kaczmarz algorithm for all sampling schemes in $\mathcal{P}_{[M]}$. This includes the sampling scheme in \cite{strohmer2009randomized}. Indeed, if we denote $p_{SV}$ the sampling scheme in \cite{strohmer2009randomized},
\[
p_{SV}(s) = \frac{\Vert a_s\Vert_2^2}{\Vert A\Vert_F^2} \quad\Rightarrow\quad \operatorname{supp}(p_{SV}) = [M],
\]
which clearly implies that $p_{SV}\in \mathcal{P}_{[M]}$. This concludes our discussion here.

As a second corollary, we demonstrate how one can produce different variations of Theorem~\ref{maintheorem1}; the following form will be used in the study of Algorithm~\ref{algorithm1}.

\begin{corollary}\label{maincor2}
Given $S\subseteq [M]$ and $\epsilon>0$, denote by $u_S$ the uniform distribution over the set $S$ and $D_{S,\,\epsilon}' = \cap_{s\in S}D_{\{s\},\,\epsilon}$, where $D_{\{s\},\,\epsilon}$'s are defined in Theorem~\ref{maintheorem1}. Consider $f(S,\,\cdot,\,u_S,\,\cdot): D_{S,\,\epsilon}' \rightarrow \mathbb{R}$ defined as in Theorem~\ref{maintheorem1}. Then, $f(S,\,x,\,u_S,\,y) \leq \alpha'(S,\,\epsilon) < 1$ for $(x,y)\in D_{S,\,\epsilon}'$. Moreover, for $s\in S$, define
\[
E_{S,s, \epsilon} = \left\{(x, y)\in D_{\{s\}, \epsilon}: x \not\in \bigcup_{s\in S} C_s\right\},
\]
where each $C_s$ is the linear subspace associated with the index $s\in S$. Then
\[
f(S,x,u_S,y) \leq \frac{\vert S\vert -1}{\vert S\vert} + \frac{1}{\vert S\vert} \alpha'\left(\{s\},\epsilon\right) = \gamma(s,\epsilon) < 1, \quad (x, y) \in E_{S, s, \epsilon}.
\]
\end{corollary}

\begin{proof}
To make our presentation clear, we assume $S = \left\{1,\dotsc, \vert S\vert \right\}$ without loss of generality. Given $s\in S$ and $\delta_s\in \mathcal{P}_{x,\,\{s\}}$, by Theorem~\ref{maintheorem1} there exists $\alpha(\{s\},\,\delta_s,\, \epsilon) < 1$ such that the function $f(\{s\},\,\cdot,\, \delta_s,\, \cdot): D_{\{s\},\,\epsilon} \rightarrow \mathbb{R}$ defined as
\[
f(\{s\},\,x,\, \delta_s,\, y) = \left\Vert \left(I - \frac{(x-P_{C_s}x)(x-P_{C_s}x)^T}{\Vert x - P_{C_s}x\Vert^2} \right) y\right\Vert^2.
\]
satisfies the bound $f(\{s\},\,x,\, \delta_s,\, y)\leq \alpha(\{s\},\,\delta_s,\, \epsilon) < 1$. In particular, we have
\[
f(\{s\},\,x,\, \delta_s,\, y)\leq \alpha(\{s\},\,\delta_s,\, \epsilon) \quad\forall (x,y)\in D_{S,\epsilon}',\, s\in S,
\]
which implies 
\begin{align*}
f(S,\,x,\,u_S,\,y) &= \sum_{s\in S} \frac{1}{\vert S\vert} \left\Vert \left(I - \frac{(x-P_{C_s}x)(x-P_{C_s}x)^T}{\Vert x - P_{C_s}x\Vert^2} \right) y\right\Vert^2 \\[5pt]
&= \sum_{s\in S} \frac{1}{\vert S\vert} f(\{s\},\,x,\, \delta_s,\, y) \leq \max_{s\in S} \left\{ \alpha(\{s\},\,\delta_s,\, \epsilon) \right\} < 1.
\end{align*}
This concludes the proof of the first statement. For the second statement, we see that $f(S,\cdot, u_S,\cdot)$ is properly defined on $E_{S,s,\epsilon}$, and
\begin{multline}\label{maineq6}
f(S,x,u_S,y) = \sum_{s'\neq s} \frac{1}{\vert S\vert} \left\Vert \left(I - \frac{(x-P_{C_{s'}}x)(x-P_{C_{s'}}x)^T}{\Vert x - P_{C_{s'}}x\Vert^2} \right) y\right\Vert^2 \\
+ \frac{1}{\vert S\vert} \left\Vert \left(I - \frac{(x-P_{C_s}x)(x-P_{C_s}x)^T}{\Vert x - P_{C_s}x\Vert^2} \right) y\right\Vert^2.
\end{multline}
Note that 
\[
\left\Vert \left(I - \frac{(x-P_{C_{s'}}x)(x-P_{C_{s'}}x)^T}{\Vert x - P_{C_{s'}}x\Vert^2} \right) y\right\Vert^2 \leq 1 \quad \forall (x,y)\in E_{S,s,\epsilon},
\]
and by the first statement
\[
\left\Vert \left(I - \frac{(x-P_{C_s}x)(x-P_{C_s}x)^T}{\Vert x - P_{C_s}x\Vert^2} \right) y\right\Vert^2 \leq \alpha'(\{s\},\epsilon) \quad\forall (x,y)\in D_{\{s\},\epsilon}, 
\]
which can then be combined with (\ref{maineq6}) to conclude the proof. 
\end{proof}

\subsection{Proof of Theorem~\ref{introthm2}, Theorem~\ref{introthm1} and Corollary~\ref{introcor1}}

\emph{Roadmap of the proof for Theorem~\ref{introthm2}: }  When the local clients run only finitely many iterations, the local updates are not exactly orthogonal projections onto a linear subspace; therefore, the idea of the proof is different from the proof of Theroem~\ref{introthm1}. For Theorem~\ref{introthm2}, the key observation is that when running only one global iteration ($\tau_g=1$), the global update is essentially randomly selecting one of the local updates. From this perspective, one can compare the scheme with the classic RK algorithm applying to a suitable overdetermined linear system to deduce the convergence. 

\begin{proof}[Proof of Theorem~\ref{introthm2}]
Denote $\mathbb{E} \left(X^{(t+1)} \vert S, X^{(t)}\right)$ the expected global position conditioning on the subset $S$ of local clients that is selected. One can observe that 
\begin{equation}\label{maineq8}
\mathbb{E} \left(X^{(t+1)} \vert S, X^{(t)}\right) = \sum_{s\in S} \frac{1}{\vert S\vert} X^{(t, T)}_s,  
\end{equation}
where $X^{(t, T)}_s$ is the position of client $s$ after $T$ local iterations. The idea of the proof is to compare the dynamics with a suitable choice of RK algorithm.

Fix $s\in [M]$. Denote $a_{s, 1},\dotsc, a_{s, m(s)}$ the rows at client $s$. Then 
\begin{equation}\label{maineq9}
\mathbb{E}\Vert X^{(t, T)}_s\Vert \leq \mathbb{E}\Vert X^{(t, 1)}_s\Vert = \frac{1}{m(s)}\sum_{i=1}^{m(s)} \left\Vert\left( I - \frac{a_{s, i}^T a_{s, i}}{\Vert a_{s,i}\Vert^2}\right)X^{(t, 0)} \right\Vert.
\end{equation}
Using (\ref{maineq9}), we have
\begin{align*}
\mathbb{E}\Vert X^{(t+1)}\Vert &= \sum_S \mathbb{E} \left(\Vert X^{(t+1)}\Vert \vert S\right) \mathbb{P}\left(S^{(t)} = S \right) \\
&= \sum_S \sum_{s\in S} \frac{1}{\vert S\vert} \mathbb{E}\Vert X^{(t, T)}_s\Vert \mathbb{P}\left(S^{(t)} = S \right) \\
&\leq \sum_S \sum_{s\in S} \frac{1}{\vert S\vert}    \frac{1}{m(s)}\sum_{i=1}^{m(s)} \left\Vert\left( I - \frac{a_{s, i}^T a_{s, i}}{\Vert a_{s,i}\Vert^2}\right)X^{(t, 0)} \right\Vert \mathbb{P}\left(S^{(t)} = S \right).
\end{align*}
One can see that the right hand side is the expected decrease when using RK to solve $A'x = 0$, where $A'$ consists of copies of $A$. One can then deduce the convergence using the convergence analysis of classic RK.   
\end{proof}

Given $S\subseteq [M]$, we consider $S_x \subseteq S$ defined as
\[
S_x = \left\{s\in S: x \not\in C_s\right\},
\]
and $u_{S_x}$ the uniform distribution over $S_x$.
Then we consider the following function
\begin{equation}\label{maineq7}
g(S,x,y) =
\begin{cases}
f(S_x,x,u_{S_x},y), \quad &\mbox{if $S_x\neq\emptyset$} \\[3pt]
1, \quad &\mbox{if $S_x = \emptyset$}
\end{cases}, \quad 
g(x, y) = \sum_{S\subseteq M} \frac{1}{{M\choose \vert S\vert}} g(S,\,x,\,y)
\end{equation}
One can see that $g(x,y)$ characterizes the convergence rate when $\tau_g=1$; for instance, see the proof of Corollary~\ref{maincor1}. We prove two lemmas, and then prove Theorem~\ref{introthm1} as a consequence of these lemmas. Let us briefly describe the proof strategy.

\emph{Roadmap of the proof for Theroem~\ref{introthm1}: } Fix a collection of linear subspaces at the local clients $\{C_s\}_{s=1}^M$. First, we demonstrate that for an overdetermined system, one can find a small enough $\epsilon>0$ such that the $\epsilon$-neighborhood of these linear subspaces has empty intersection on the unit sphere; this is done in Lemma~\ref{mainlem1}. Using Corollary~\ref{maincor2}, we know that if the initial position at a given federated round is $\epsilon$ away from some linear subspace $C_s$, then we are guaranteed to shrink the distance between our current position and the solution by a constant factor via projection onto $C_s$. By Lemma~\ref{mainlem1}, we know that an arbitrary point on the unit sphere is at least $\epsilon$ away from some linear subspace. At each federated round, there is always some chance to select a client $s$ with $C_s$ being $\epsilon$-away from the normalized global position, and we are guaranteed to shrink the current position by a uniform factor if we orthogonally project onto $C_s$. Using this fact, we show that on average we can shrink our current position by a constant factor regardless of our initial position at a given federated round; this is done in Lemma~\ref{mainlem2}. Finally, everything is put together with the scaling argument in Corollary~\ref{maincor1} to obtain Theorem~\ref{introthm1}. To deduce Corollary~\ref{introcor1} for the underdetermined system, the key observation is to decompose the initial position as $x = x_1 + x_2$, where $x_1\in C = \cap_s C_s$ and $x_2\in C^{\perp}$. One can see that $x_1$ is a fixed point for the orthogonal projections onto $C_s$'s, and that the orthogonal projections onto $C_s$'s restricted to $C^{\perp}$ reduce to a determined system; these observations allow us to deduce Corollary~\ref{introcor1} from Theorem~\ref{introthm1}.

\begin{lemma}\label{mainlem1}
Let $\{C_s\}_{s=1}^M$ be a collection of linear subspaces such that $\cap_{s=1}^M C_s = \{0\}$. Then there exits $\epsilon > 0$ such that 
\[
\bigcap_{s\in M} \left\{ x\in S^{n-1}: d(x,\,C_s\cap S^{n-1}) < \epsilon \right\} = \emptyset.
\]
\end{lemma}

\begin{proof}
Define $\tilde{d}:S^{n-1}\rightarrow \mathbb{R}$ as
\[
\tilde{d}(x) = \max_{s\in S^{n-1}} d(x, C_s\cap S^{n-1});
\]
one can see that $\tilde{d}$ is continuous. Since $\cap_{s=1}^M C_s \cap S^{n-1} = \emptyset$, one has
\[
\tilde{d}(x) > 0 \quad\forall x\in S^{n-1}.
\]
This further implies that there exists $\epsilon > 0$ such that $\tilde{d}(x)\geq \epsilon$ for all $x\in S^{n-1}$ by compactness of $S^{n-1}$. This concludes our proof. 
\end{proof}

\begin{lemma}\label{mainlem2}
Given $0 < \epsilon \ll 1$, denote $K_{s,\epsilon} = \{x\in S^{n-1}: d(x, C_s\cap S^{n-1}) \geq \epsilon\}$ for all $ s\in [M]$. Then $g: S^{n-1}\times S^{n-1}\rightarrow \mathbb{R}$ defined in (\ref{maineq7}) satisfies
\[
g(x,y) \leq \beta < 1, \quad\forall (x,y)\in \bigcup_{s\in M} \left( K_{s,\epsilon}\times K_{s,\epsilon}\right).
\]
\end{lemma}

\begin{proof}
Given $x\in K_{s,\epsilon}$, one has
\[
s \in S \quad\Rightarrow\quad s \in S_x,
\]
which further implies
\[
g(x,y) = \sum_{S\subseteq M} \frac{1}{{M\choose \vert S\vert}} g(S,x,y) = \sum_{s\in S} \frac{1}{{M\choose \vert S\vert}} f(S_x,x,u_{S_x},y) + \sum_{s\not\in S} \frac{1}{{M\choose \vert S\vert}} g(S,x,y).
\]
By Corollary~\ref{maincor2}, one has 
\[
f(S_x,x,u_{S_x},y) \leq \gamma_{s,\epsilon} \quad\forall (x,y)\in E_{S_x,s,\epsilon},
\]
which can be combined with the fact that there are ${M-1\choose\vert S\vert -1}$ subsets of $M$ containing $s$ to get
\[
g(x,y) \leq \frac{{M-1\choose\vert S\vert -1}}{{M\choose \vert S\vert}} \gamma_{s,\epsilon} + \frac{{M\choose \vert S\vert}-{M-1\choose\vert S\vert -1}}{{M\choose \vert S\vert}} = \beta_{s,\epsilon} < 1, \quad (x,y)\in K_{s,\epsilon}\times K_{s,\epsilon}.
\]
One can then take $\beta_{\epsilon} = \max_{s\in M} \{\beta_{s,\epsilon}\}$ to conclude the proof.
\end{proof}

Now, we combine Lemma~\ref{mainlem1} and Lemma~\ref{mainlem2} to prove Theorem~\ref{introthm1}.

\begin{proof}[Proof of Theorem~\ref{introthm1}]
First, we assume that the server runs one global iteration ($\tau_g = 1$). Let $\epsilon>0$ be small enough so that Lemma~\ref{mainlem1} holds. For $K_{s,\epsilon}$ defined in Lemma~\ref{mainlem2}, one has
\[
X\setminus\left(\bigcup_{s\in [M]}K_{s,\epsilon}\right) = \bigcap_{s\in [M]}K_{s,\epsilon}^c = \emptyset
\]
by Lemma~\ref{mainlem1}; this shows that $\{K_{s,\epsilon}\}_{s\in [M]}$ covers $S^{n-1}$, and therefore 
\[
(x, x)\in \cup_{s\in [M]}\left( K_{s,\epsilon} \times K_{s,\epsilon}\right) \quad\forall x\in S^{n-1}.
\]
By Lemma~\ref{mainlem2}, one sees that
\[
\mathbb{E}\left( \left. \Vert X^{(t+1)}\Vert^2 \, \right\vert X^{(t)}\right) = g(X^{(t)}, X^{(t)}) \leq \beta_{\epsilon} < 1, \quad\mbox{if}\quad \Vert X^{(t)}\Vert = 1.
\]
One can then iterate as in the proof Corollary~\ref{maincor1} to conclude the proof. For general $\tau_g\in\mathbb{N}$, we observe that the operator norm of an orthogonal projection is less than or equal to $1$, and therefore the convergence result for general $\tau_g\in\mathbb{N}$ follows from $\tau_g=1$. 

\end{proof}

Finally, we demonstrate how one can deduce the convergence behavior for underdetermined systems.

\begin{proof}[Proof of Corollary~\ref{introcor1}]
Given an initial $X^{(0)}\in\mathbb{R}^n$, we consider the orthogonal decomposition $X^{(0)}= P_{C^{\perp}} X^{(0)} + P_{C} X^{(0)}$, where $C$ is defined as in the statement of Corollary~\ref{introcor1}. Then one has
\begin{align}
X^{(1)} &= \sum_{s\in S^{(0)}} \frac{1}{\vert S^{(0)}\vert} \left(I - \frac{(X^{(0)}-P_{C_s}X^{(0)})(X^{(0)}-P_{C_s}X^{(0)})^T}{\Vert X^{(0)} - P_{C_s}X^{(0)}\Vert^2} \right) X^{(0)} \\
&= \sum_{s\in S^{(0)}} \frac{1}{\vert S^{(0)}\vert} \left(I - \frac{(X^{(0)}-P_{C_s}X^{(0)})(X^{(0)}-P_{C_s}X^{(0)})^T}{\Vert X^{(0)} - P_{C_s}X^{(0)}\Vert^2} \right) P_{C^{\perp}} X^{(0)} + P_{C} X^{(0)},
\end{align}
since $C\subseteq C_i$ for all $i\in\{1,\dotsc, M\}$ and $P_CX^{(0)}$ is a fixed point for projections onto $C_i$'s. On the other hand, if we focus on the linear subspace $P_{C^{\perp}}(\mathbb{R}^n)$, $C^{\perp}\cap C = \{0\}$, and it is reduced to a determined system, and we can apply Theorem~\ref{introthm1} to deduce
\[
\Vert X^{(t+1)} - P_CX^{(0)}\Vert^2 \leq \beta^{t+1} \Vert P_{C^{\perp}}X^{(0)}\Vert^2 \leq \beta^{t+1} \Vert X^{(0)}\Vert^2, \quad t= 0,1,\dotsc.
\]
for some $0 < \beta < 1$.
\end{proof}

\section{Experiments}\label{experiments}

In the following, we demonstrate experimentally the convergence of FedRK, and we demonstrate its behavior when applied to sparse approximation problems and least squares problems. Finally, we apply Algorithm~\ref{algorithm3} to real data and compare the result with Lasso.

\subsection{FedRK}
In this section, we consider solving an overdetermined consistent system. We consider $A\in\mathbb{R}^{2048\times 1024}$, $x^*\in\mathbb{R}^{1024}$ and $b=Ax^*\in\mathbb{R}^{2048}$, where the entries of $A$ and $x$ are generated as i.i.d. standard Gaussian. The data is distributed evenly to $M=16$ clients. At each federated round, five local clients are selected to participate in the update, and each of them run $\tau$ local iterations. The server runs 20 global iterations after receiving the local updates. This demonstrates the convergence behavior of Algorithm~\ref{algorithm1}, and the results suggest that running more local iterations may improve the convergence rate. The results are in Figure~\ref{experimentsfig1}.

\begin{figure}[ht]
    \centering
    \includegraphics[width=119mm]{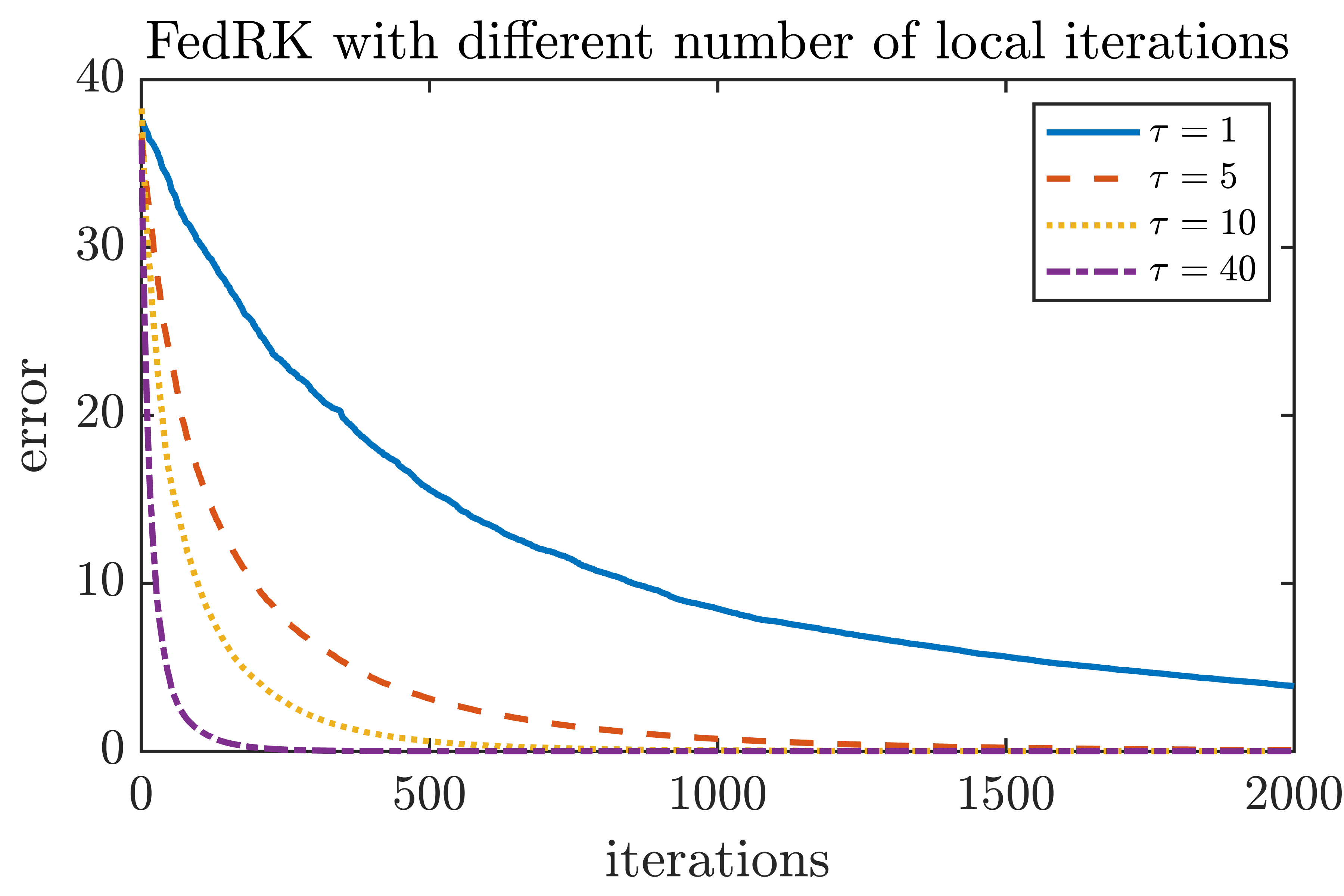}
    \caption{FedRK with clients running different numbers of local iterations}
    \label{experimentsfig1}
\end{figure}

\subsection{Sparse approximation problems}

In this section, we consider the sparse signal recovery problem. For certain types of signals, one can find a good basis so that these signals can be approximated by sparse representations, and an important question is how one can reconstruct the sparse approximations \cite{needell2009cosamp}. One type of algorithm for this problem is the iterative hard thresholding algorithm \cite{blumensath2008iterative}. Given a sparsity level $s\in\mathbb{N}$, the objective is to solve
\[
\min_{x\in\mathbb{R}^n} \Vert b -Ax\Vert_2^2 \quad\mbox{subject to}\quad \Vert x\Vert_0 \leq s, 
\]
where $\Vert\cdot \Vert_0$ counts the number of non-zero entries, and the iterative hard thresholding algorithm performs gradient descent followed by hard thresholding (projection onto the $s$ largest entries) in each iteration. Recently, it has been proposed to replace the gradient descent step with a RK projection \cite{jeong2023linear, zhang2015iterative}. Our experiment demonstrates such a strategy can be extended to the federated setting. We consider the measurement matrix $A\in\mathbb{R}^{256\times 1024}$, $x^*\in \mathbb{R}^{1024}$ the target sparse solution, $e\in \mathbb{R}^{256}$ the measurement noise, and $b= Ax^* + 0.01e\in\mathbb{R}^{256}$ the perturbed observation; the entries in $A$ and $e$, and the non-zero entries of $x_*$ are all generated from the standard Gaussian. Data is distributed evenly to $16$ clients, and at each federated round, $5$ local clients participate. Each of the local clients runs $20$ local RK iterations, and the global server runs $20$ global iterations after receiving the local updates. The sparsity level in this experiment is $9$, and we demonstrate that Algorithm~\ref{algorithm3} can recover the true signal. The experiment is performed 50 times with different initializations and sampling schemes, and we record the number of times each feature is selected in Figure~\ref{experimentsfig2}.

\begin{figure}[ht]
    \centering
    \includegraphics[width=119mm]{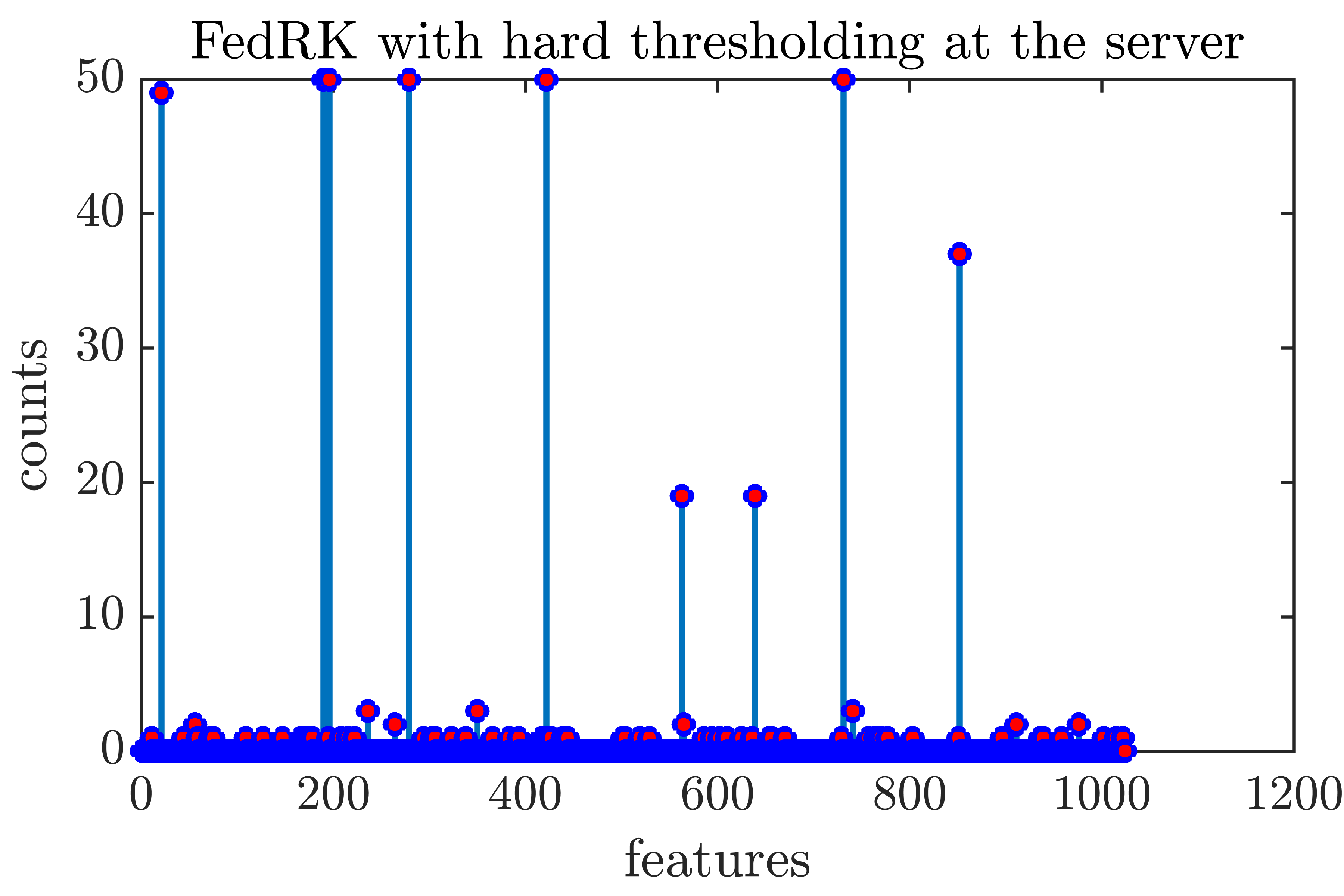}
    \caption{The number of times each feature was selected in the $50$ trials when Algorithm~\ref{algorithm3} with sparsity level $s=9$ was applied.}
    \label{experimentsfig2}
\end{figure}

\subsection{Least squares} 
While the main objective of this paper is to solve consistent linear system, in many real‐world applications the right‐hand side $b\in\mathbb{R}^m$ is corrupted by noise or modeling error and one instead faces an inconsistent system. To assess the robustness of our method under these more realistic conditions, we consider general least squares problems. When the linear system is inconsistent, recall that classical RK does not converge to the least squares solution; however, its iterates do eventually remain within some horizon of the least squares solution \cite{needell2010randomized}. We consider $A\in\mathbb{R}^{2048\times 256}$ and $b\in\mathbb{R}^{2048\times 1}$, where the entries in $A$ and $b$ are generated from standard Gaussian. Our aim is to solve $\min_x\Vert Ax-b\Vert^2$; data are distributed evenly to $M=16$ clients. In each federated round, all local clients participate. We apply FedRK to the system, and the algorithm converges to a horizon of the least squares solution; moreover, by adding a suitable amount of noise one can shrink the convergence horizon. Specifically, we expand the matrix $A$ by adding gaussian noise columns,
\[
A' = \begin{bmatrix} A & B
\end{bmatrix},
\]
where the entries in $B$ are generated as i.i.d. standard Gaussian, and we apply FedRK to solve the extended system $\min_x \Vert A'x-b\Vert^2$; note that the perturbation $B$ can be added at the clients level, and the strategy is suitable for the federated setting. To motivate our strategy, it is known that in high-dimensional space, two independent Gaussian vectors are almost orthogonal with high probability, and therefore the columns in $B$ are likely to be almost orthogonal to the column space of $A$; hence, the columns in $B$ capture the components of $b$ that is perpendicular to the column space of $A$. This is inspired by the randomized extended Kaczmarz algorithm \cite{zouzias2013randomized}, where column operations are involved to solve the least squares problem; however, column operations are not suitable in the federated setting, and therefore, we create the extra matrix $B$ instead. This is motivated to the approach taken in \cite{epperly2024randomized}, where the Monte Carlo method is combined with RK to solve the least squares problem. The strategy proposed in \cite{epperly2024randomized} does not involve column operations, and it would be interesting to see if it can be extended to the federated setting. The results are summarized in Figure~\ref{experimentsfig3}.

\begin{figure}[ht]
    \centering
    \includegraphics[width=119mm]{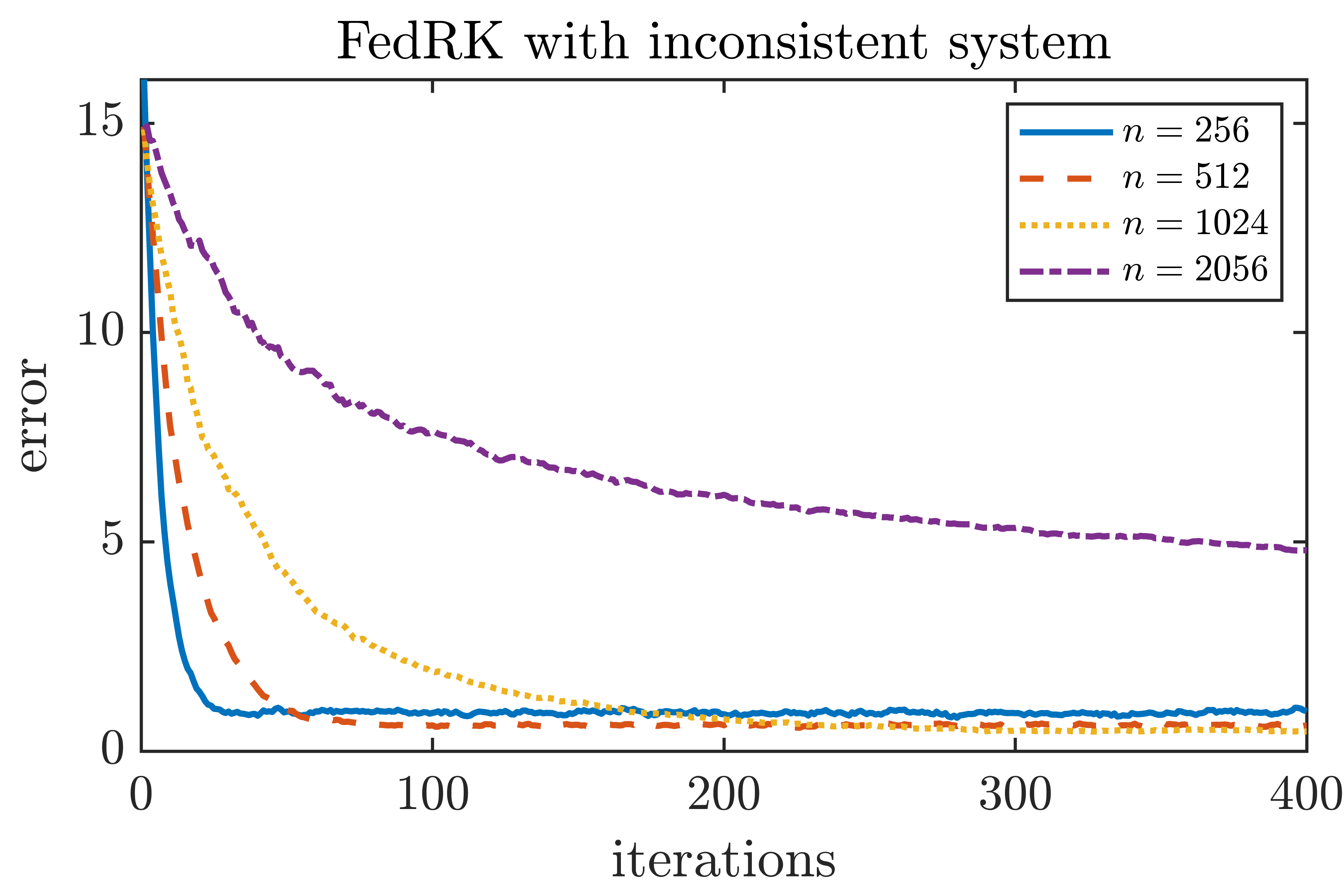}
    \caption{This experiment applies Algorithm~\ref{algorithm1} to the augmented system $A'x=b$ to solve the least squares problem, where $A'$ is obtained from $A$ by augmenting $n-256$ noisy columns. The result shows that we can shrink the convergence horizon by adding a suitable amount of noisy columns.}
    \label{experimentsfig3}
\end{figure}

\subsection{Prostate cancer data}

In this section, we test our algorithm on the prostate cancer data \cite{stamey1989prostate}, which was used in \cite{tibshirani1996regression}. The features are 1.~intercept (intcpt), 2.~log(cancer volume) (lcavol), 3.~log(prostate weight) (lweight), 4.~age, 5.~log(benign prostatic hyperplasia amount) (lbph), 6.~seminal vesicle invasion (svi), 7.~log(capsular penetration) (lcp), 8.~Gleason score (gleason) and 9.~percentage Gleason scores 4 or 5 (pgg45), and the target variable is log(prostate specific antigen) (lpsa). We adapt a more recent use of this data \cite{hastie2009elements}, where the Lasso is fit to the data. Let us briefly recall that given the data matrix $X\in\mathbb{R}^{N\times p}$ and the target vector $y\in\mathbb{R}^N$, Lasso solves 
\[
\hat{\beta} = \argmin_{\beta} \frac{1}{2}\sum_{i=1}^N (y_i - \sum_{j=1}^px_{ij}\beta_j)^2 + \lambda\sum_{j=1}^p \vert\beta_j\vert,
\]
where $\lambda\in\mathbb{R}_{\geq 0}$ is the regularizing parameter (here we assume that the intercept is included as the first column of $X$); it is known that when $\lambda$ is large, $\hat{\beta}$ tends to be sparse, and a feature is said to be selected if it has non-zero coefficients. In this case 4 features are selected: lcavol, lweight, lbph and svi for some chosen $\lambda\in\mathbb{R}_{\geq 0}$. Following \cite{hastie2009elements}, we first standardize the feature columns, then add an extra column of ones to represent the intercept; we distribute the data evenly to 7 local clients. To compare with the result in \cite{hastie2009elements}, we set the sparsity level to 5 (so that we include one for the intercept). At each federated round, 3 local clients participate and each of them runs 20 iterations using the local data; after receiving the local updates, the server runs 20 iterations of RK and then apply the hard thresholding operator $T_5$. We performed 2000 federated rounds, and counted the number of times each feature has a non-zero coefficient after hard thresholding. The result is recorded in Figure~\ref{experimentsfig4}, and the distribution is consistent with the analogous selection of Lasso if we look at the top 5 features.

\begin{figure}[ht]
    \centering
    \includegraphics[width=119mm]{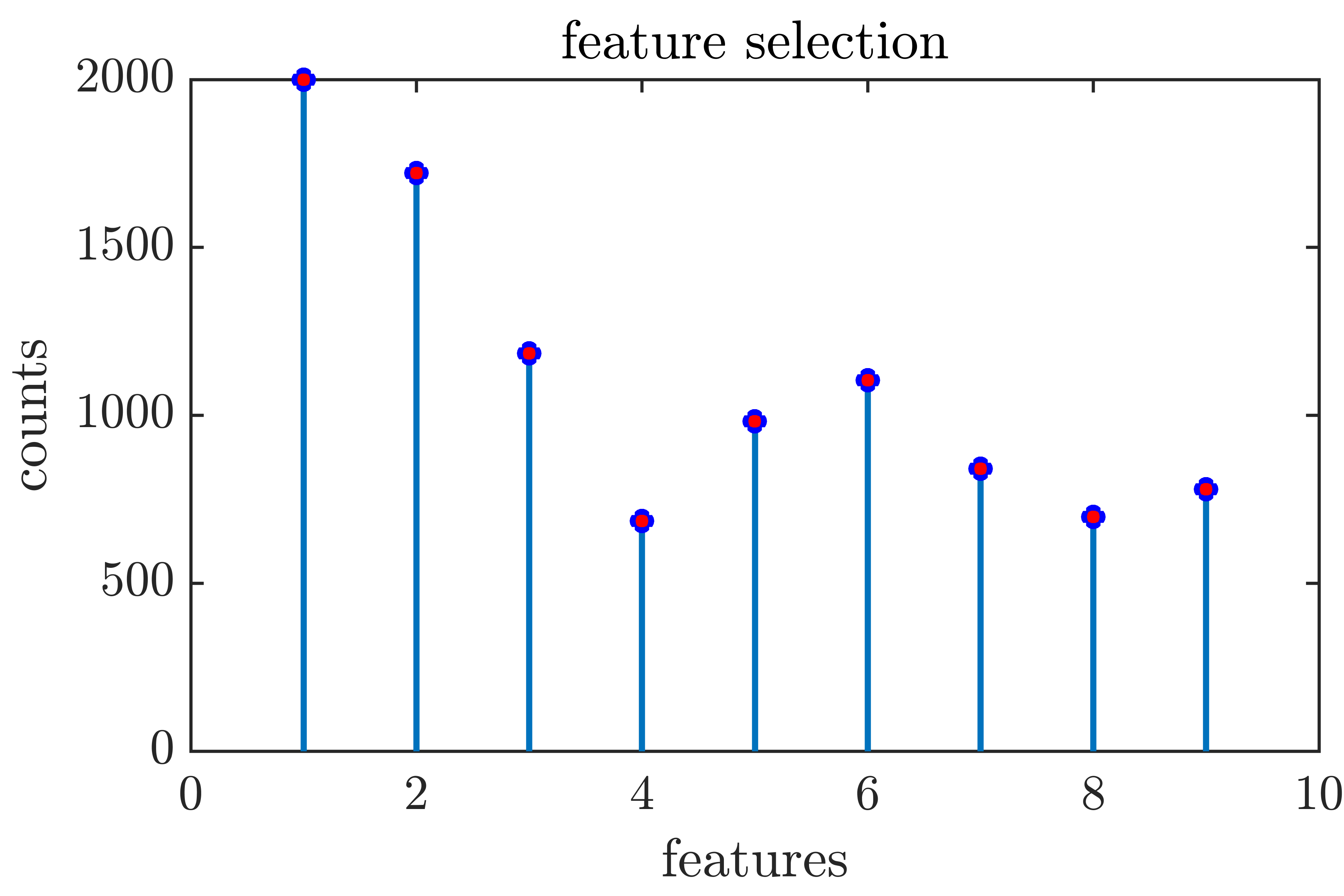}
    \caption{feature selection via Algorithm~\ref{algorithm3}}
    \label{experimentsfig4}
\end{figure}

\section{Discussion}\label{discussion}
In this paper, we proposed a federated algorithm (Algorithm~\ref{algorithm1}) for solving large linear systems and derived its convergence property. When applied to inconsistent systems, our experiments showed that it converges to a horizon of the least squares solution. We also proposed a modified version of our algorithm (Algorithm~\ref{algorithm3}) for sparse approximation problems. We applied Algorithm~\ref{algorithm3} to real data and showed that it has potential use for feature selection in the federated setting. For future work, it would be interesting to extend the approach here to more general optimization problems.  

\bmhead{Acknowledgements}
DN was partially supported by NSF DMS 2408912.

\bibliography{sn-bibliography}

\end{document}